\documentclass[twoside,a4paper,12pt]{amsart}
\usepackage{amsfonts, amsbsy, amsmath, amsthm, amssymb, latexsym}
\usepackage{mathrsfs, mathabx}
\usepackage[top=34mm,right=20mm,bottom=34mm,left=20mm]{geometry}
\usepackage{bm}
\usepackage{fancyhdr}
\usepackage{enumerate}
\usepackage{tikz, tikz-cd, tkz-euclide}
\usetikzlibrary{decorations.markings,positioning}
\usepackage{xcolor}
\usepackage{faktor}

\usepackage[official]{eurosym}
\usepackage{array,float,graphicx,epic,eepic}
\usepackage{amssymb, amsmath, amsthm}
\usepackage{changepage}
\usepackage{comment}
\usepackage{thmtools,thm-restate}

\usepackage[numbers]{natbib}
\usepackage{url}
\renewcommand{\a}{\alpha}

\newcommand{\R}{\mathbb{R}}

\renewcommand{\bf}{\textbf} 
\renewcommand{\l}{\lambda} 
 
 \renewcommand{\to}{\rightarrow}

\newtheorem{thm}{Theorem}[section]
\newtheorem{prop}[thm]{Proposition}
\newtheorem{lem}[thm]{Lemma}
\newtheorem{cor}[thm]{Corollary}

\theoremstyle{definition}
\newtheorem{defn}[thm]{Definition}
\newtheorem*{rem}{Remark}
\newtheorem*{nota}{Notation}

\title{Fixed points of irreducible, displacement one automorphisms of free products}

\author{Matthew Collins}

\begin{document}
\begin{abstract}
	We consider the action of outer automorphisms on the deformation space $\mathcal{O}$ of $G$-trees given by a free product decomposition of a group $G$. We show that an irreducible, displacement 1 automorphism fixes exactly one point in $\mathcal{O}_1$ (the covolume 1 slice of $\mathcal{O}$).
\end{abstract}
	
\maketitle

\tableofcontents

\section{Introduction}

This paper can be thought of as a generalisation of a paper by Dicks \& Ventura \cite{dicksventura}, in which the authors classify the irreducible, growth rate 1 automorphisms of free groups $F_n$. In the process of doing so, they show that each of these automorphisms can be represented by a graph automorphism of a graph with fundamental group $F_n$. When combined with the results of \cite{bestvinahandel} and \cite{francavigliamartino2015}, this means that an irreducible, growth rate 1 automorphism of a free group $F_n$ fixes a single point in Culler-Voghtmann space $CV_n$. The main result of this paper is a generalisation of this result: we prove that an irreducible, growth rate 1 automorphism of a free product $G=G_1*\ldots* G_k* F_r$ fixes a single point in the deformation space $\mathcal{O}_1$.

The free group version of this result is stated explicitly in \cite[p.10, Thm 3.8]{francavigliamartino2021}, and it follows from Dicks \& Ventura's classification like so: Every irreducible outer automorphsim of $F_n$ is topogically represented by an irreducible train track map $f$ on a graph in Culler-Vogtmann space $CV_n$ \cite[p.9, Thm 1.7]{bestvinahandel}. If the automorphism is growth rate 1, then $f$ is a finite order homeomorphism -  in this case, a graph automorphism. Thus the graph automorphisms found by Dicks \& Ventura in \cite{dicksventura} are in fact train track maps.

The minimally displaced set in $CV_n$ of an irreducible automorphism coincides exactly with the set of points which support train track maps \cite[p.32, Thm 8.19]{francavigliamartino2015}. Additionally, it can be shown that the growth rate of an irreducible automorphism is equal to its displacement, and if the automorphism has displacement 1, then the minimally displaced set is actually a fixed point set. It follows that the graphs supporting the train track maps found by Dicks \& Ventura are fixed points in Culler-Voghtmann space.

The final step is to show the uniqueness of these points, which is done in the proof of \cite[p.10, Theorem 3.8]{francavigliamartino2021}.

Our generalisation to free products follows a similar outline - however, we instead use the deformation space $\mathcal{O}$, otherwise known as \emph{outer space}, which is a generalisation of Culler-Vogtmann space to free products $G=G_1*\ldots*G_k*F_r$. The notion of a deformation space was first introduced by Forester \cite{forester}, and they have since been studied in \cite{cullervogtmann} and \cite{cullermorgan}. Given a group $G$, one considers minimal, cocompact, isometric actions of $G$ on metric simplicial trees. These trees, together with their actions, are called $G$-trees. Two $G$-trees are said to be equivalent if there exists an equivariant isometry between them, and one defines $\mathcal{O}$ to be the space of equivalence classes of $G$-trees which share the same set of elliptic subgroups - that is, subgroups which fix a point in the tree.

The group of outer automorphisms which preserve the set of conjugacy classes $\{[G_1],\ldots,[G_k]\}$ acts on $\mathcal{O}$ by ``twisting'' the actions of the $G$-trees. This group is denoted $\text{Out}(G,\mathcal{G})$, and we study its action on the covolume one slice of $\mathcal{O}$ (denoted $\mathcal{O}_1$) by using the asymmetric Lipschitz metric: For any two $G$-trees $T,S\in\mathcal{O}_1$, we write $\Lambda_R(T,S)$ to denote the asymmetric Lipschitz distance, or stretching factor, between them.

For an automorphism $\a\in\text{Out}(G,\mathcal{G})$, one can define the displacement of $\a$ as $\lambda_\a = \inf\{\Lambda(T,\a T)\mid T\in \mathcal{O}_1\}$. The \emph{minimally displaced set} of $\a$, $\text{Min}_1(\a)$, is the set of $G$-trees $T$ in $\mathcal{O}_1$ which realise this infimum. It can shown that if $\a$ is irreducible, then the displacement $\lambda_\a$ is not just an infimum, but it is a minimum, and hence the set $\text{Min}_1(\a)$ is non-empty. In addition, it can be shown that $\Lambda_R(T,S) = 1$ if and only if $T$ and $S$ represent the same equivalence class in $\mathcal{O}_1$ - hence $\lambda_\a = 1$ implies that $\text{Min}_1(\a)$ is the fixed-point set of $\a$.

One can think of $\mathcal{O}_1$ as a union of open simplices, where a $G$-tree's position in its simplex is determined by the lengths of its edges. In Theorem \ref{thm:permuting_edge_orbits}, we show that the action of $\a$ on each $T\in\text{Min}_1(\a)$ can be toplogically represented by an isometry of $T$, and that this isometry must cyclically permutes the $G$-orbits of edges in $T$. It follows that the edges of $T$ must all have the same length, and hence $T$ must lie at the centre of its open simplex - thus $\text{Min}_1(\a)$ must consists solely of simplex centres. However, it is shown in \cite[p.19, Cor 5.4]{francavigliamartino2018b} that $\text{Min}_1(\a)$ is connected by so-called simplicial paths. Since a nontrivial simplicial path between the centres of two simplices must pass through a point which is not at the centre of a simplex, our main result follows:

\begin{restatable*}{thm}{MSonepoint}\label{thm:min_set_single_point}
	Let $\a\in\text{Out}(G,\mathcal{G})$ be irreducible and displacement 1. Then $\text{Min}_1(\a) = \text{Fix}_1(\a)$ is a single point.
\end{restatable*}

\section{Groups acting on trees}

For the duration of this chapter, let $G$ be a group.

\subsection*{- Metric simplicial trees}

\begin{defn}
	An \emph{$\R$-tree} is a non-empty metric space in which any two points are joined by a unique arc, and in which every arc is isometric to a closed interval in the real line.
\end{defn}

\begin{defn}
	Let $p$ be a point in a non-trivial $\R$-tree $T$.
	\begin{itemize}
		\item $p$ is called a \emph{branch point} if $T-p$ has three or more components.
		\item $p$ is called \emph{regular} if $T-p$ has exactly two components.
		\item $p$ is called \emph{external} otherwise
	\end{itemize}
	Points which are not regular are called \emph{non-regular}.
\end{defn}

\begin{defn}
	A \emph{metric simplicial tree} is an $\R$-tree whose set of non-regular points is discrete.
\end{defn}

It will be useful to give these metric simplicial trees a combinatorial structure. Let $C$ be a $1$-dimensional simplicial complex. The 1-simplices will be called \emph{edges}, and the 0-simplices will be called \emph{vertices}. One can construct metric simplicial trees from simplicial complexes as follows:

\begin{defn}
	Let $C$ be a 1-dimensional simplicial complex. The \emph{geometric realisation} of $C$ is the metric space obtained from $C$ by assigning the length of every edge in $C$ to be 1. We give the geometric realisation the path metric topology.
\end{defn}

\begin{thm}\label{thm:G-trees_come_from_complexes}
	An $\R$-tree $T$ is metric simplicial if and only if it is homeomorphic to the geometric realisation of a connected $1$-dimensional simplicial complex $C$ with trivial fundamental group.
\end{thm}

\begin{rem}
	There exists an alternative way to view this construction. Let $T$ and $C$ be as in Theorem \ref{thm:G-trees_come_from_complexes}. Then one can think of $T$ as being obtained from $C$ by assigning a length $L(e)$ (not necessarily 1) to every edge $e$ in $C$, and to ensure discreteness of the non-regular points we impose the condition that for every vertex $v\in C$, $\text{inf}\{L(e)\mid e \text{ is incident to } v\}> 0$.
\end{rem}

It is important to note that if $T$ is a metric simplicial tree, then the geometric realisation to which $T$ is homeomorphic to is not unique: Dividing any edge of $C$ into two by adding a new vertex will result in a new simplicial complex whose geometric realisation is also homeomorphic to $T$. Thus we have a degree of choice over the structure of our trees. Once a choice of $C$ has been made, we shall simply say ``$e$ is an edge of $T$" to mean that $e$ is an edge of $C$, and similar for vertices.

When $T$ is acted upon by a group $G$, the action can be used to determine our structure, as described in the next section. The conventional way of doing this, which we shall also be using, is described in the next section.

\subsection*{- $G$-trees}\label{sec:1.2}

\begin{defn}
	Let $T$ be a metric simplicial tree with an underlying simplicial complex. We define a  \emph{subforest} of $T$ to be a subspace given by a set $E'$ of edges and a set $V'$ of vertices, such that the incident vertices of every edge in $E'$ lie in $V'$.
	
	A \emph{subtree} of $T$ is a connected subforest.
\end{defn}

What this definition means is that we are defining subforests and subtrees so that they respect the underlying simplicial complex. In general this is not required, but it suits our purposes for this paper.

\begin{defn}\label{defn:many_adjectives}
	Suppose $G$ acts on a metric simplicial tree $T$ with an underlying simplicial complex.
	\begin{itemize}
		\item The action is said to be \emph{simplicial} if it maps vertices to vertices and edges to edges.
		
		\item If no edge of $T$ is sent to its inverse by any element of $G$, we say that \emph{$G$ acts without inversions}.
		
		\item We say that both $T$ and the action are \emph{minimal} if $T$ contains no proper, $G$-invariant subtree.
		
		\item Let $x\in T$. The \emph{stabiliser $\text{stab}(x)$ of $x$} is defined to be the subgroup $\{g\mid x\cdot g = x\}$ of $G$.
		
		Let $e$ be an edge of $T$. Then we define the the \emph{stabiliser $\text{stab}(e)$ of $e$} to be the subgroup of $G$ which fixes $e$ but does not necessarily preserve the orientation of $e$.
		
		If every edge in $T$ has trivial stabiliser, we say that $T$ is \emph{edge-free}.
		
		Let $p$ be a vertex in $T$. If $\text{stab}(p) = 1$, we say that $p$ is \emph{free}. Otherwise, it is \emph{non-free}.
	\end{itemize}	
\end{defn}

\begin{defn}\label{defn:G-tree}
	A \emph{$G$-tree} is a triple $(T,d_T,\cdot)$, where $T$ is a metric simplicial tree, $d_T$ is the metric on $T$, and $\cdot$ is an isometric group action $T\times G \to T,\ (x,g)\mapsto x\cdot g$.
	
	If the metric and action are obvious from context, we may choose to omit one or both of  them from the notation.
\end{defn}

\begin{rem}
	We have chosen to define $G$-trees with a right action so that $\forall x\in T$, $\forall g\in G$, and $\forall H\leq G$, $\text{stab}(x\cdot g) = \text{stab}(x)^g$ and $\text{Fix}(H^g) = \text{Fix}(H)\cdot g$.
	
	Had we chosen to act on the left, acting by $g$ would have caused the stabilisers to be conjugated by $g^{-1}$, and similar for the fixed point sets.
\end{rem}

We are now ready to choose a simplicial structure for our $G$-trees. Let $T$ be a $G$-tree. Then the \emph{simplest} structure on $T$ - that is, the structure containing the fewest vertices - is obtained by defining the vertex set to be the set of non-regular points of $T$. One then takes the edge set to be the set of simple arcs between elements of the vertex set which do not contain any other vertices.

Using the simplest structure, the action of $G$ on $T$ is simplicial. However, some edges may be sent to their inverses by elements of $G$. Future calculations will be easier if we have an action without inversions, therefore we shall instead use the following structure:

\begin{itemize}
	\item We define the vertex set to be the set of non-regular points of $T$, together with the midpoints of all the edges of the simplest structure which were inverted by an element of $G$. We denote this vertex set by $V(T)$.
	
	\item We then define the edge set to be the set of simple arcs between elements of the vertex set which do not contain any other vertices. We denote this edge set by $E(T)$.
\end{itemize}

Essentially, we divide each inverted edge into two edges by placing a new vertex at its midpoint. With this new structure the action is still simplicial and, in addition, it is without inversions. This simplicial structure is used by we shall be giving to all $G$-trees throughout this paper.

\begin{rem}
	$V(T)$ as defined above is a discrete set, and hence is a well defined vertex set.
\end{rem}

The following two propositions follow immediately from this choice of simplicial structure.

\begin{prop}\label{prop:no_degree_1}
	Minimal $G$-trees do not contain any degree 1 vertices (and hence the set of non-regular points is exactly the set of branch points).
\end{prop}



\begin{prop}\label{prop:no_free_degree_2}
	$G$-trees do not contain any free vertices of degree 2.
\end{prop}

\subsection*{- Equivalence}

\begin{defn}
	Let $(T,d_T,\cdot),\ (S,d_S,*)$ be $G$-trees. We say a map of trees $f:T\to S$ is a $G$-\emph{equivariant} map from $(T,d_T,\cdot)$ to $(S,d_S,*)$ if $f(x\cdot g) = f(x)*g$ for all $x\in T$, for all $g\in G$.
\end{defn}

\begin{defn}
	Two $G$-trees $(T,d_T,\cdot )$, $(S,d_S,*)$ are said to be \emph{equivalent} if there exists a $G$-equivariant isometry between them. We write $(T,d_T,\cdot)\sim (S,d_S,*)$ to denote equivalence.
\end{defn}

\begin{defn}
	We say a map of $G$-trees is \emph{simplicial} if it maps vertices to vertices. Note that it does not have to map edges to edges, and hence this definition differs from that of a simplicial group action.
\end{defn}

\section{Bass-Serre Theory}

\subsection*{- Graphs of Groups}

\begin{defn}\cite[p.113]{cohen}
	A \emph{graph} $Y$ consists of the following:
	\begin{itemize}
		\item Two disjoint sets $V(Y)$ and $E(Y)$, called the \emph{vertex} and \emph{edge} sets of $Y$ respectively.
		\item A function $\overline{\phantom{e}}: E(Y)\to E(Y)$ such that, for all $e\in E(Y)$, $\overline{e}\neq e$ and $\overline{\overline{e}} = e$.
		\item A function $\iota: E(Y)\to V(Y)$, and another function $\tau: E(Y)\to V(Y)$ defined by $\tau e := \iota \overline{e}$. We call $\iota e$ the \emph{initial vertex} of $e$, and $\tau e$ the terminal vertex of $e$.
	\end{itemize}
	We say $Y$ is \emph{finite} if $V(Y)$ and $E(Y)$ are both finite.
\end{defn}

Graphs defined in this way - by considering each unoriented edge as a pair of oriented edges $(e,\overline{e})$ - are often referred to as \emph{Serre graphs}.

\begin{defn}
	A \emph{metric graph} is a graph $Y$ together with a length function $L:E(Y)\to \R$ such that, for all edges $e$ of $Y$, $L(e)=L(\overline{e})$. This length function induces a metric $d_Y$ on $Y$.
\end{defn}

Diverting briefly back to the previous chapter, we remark that a 1-dimensional simplicial complex can be thought of as a Serre graph by considering each 1-simplex to be an edge pair. Thus $G$-trees (and their quotients) can be thought of as metric graphs, and depending on context we may treat them as such. This allows us to make the following observations:

\begin{prop}\label{prop:G_acts_via_g.autos}
	$G$ acts on $G$-trees via graph automorphisms (without inversions).
\end{prop}




\begin{prop}\label{prop:finite_iff_compact}
	Let $T$ be a metric simplicial tree. Then $T/G$ is finite if and only if it is compact (under the path metric topology).
\end{prop}



We now return to defining graphs of groups.

\begin{defn}\cite[p.198]{cohen}
	A \emph{graph of groups} $X$ consists of:
	\begin{itemize}
		\item[(i)] An connected graph $Y$
		\item[(ii)] A group $G_v$ for each vertex $v$ of $Y$, and a group $G_e$ for each edge $e$ of $Y$ such that $G_{\overline{e}} = G_e$.
		\item[(iii)] For each edge $e$ of $Y$, a monomorphism $\rho_{e}:G_{e}\to G_{\tau e}$, where $\iota e$ and $\tau e$ are the endpoints of $e$.
	\end{itemize}
	If $Y$ is a metric graph, then we say $X$ is a \emph{metric} graph of groups.
\end{defn}

Let $X$ be a graph of groups on a graph $Y$. One can define the fundamental group of $X$ in a similar manner to that of a standard graph, by thinking of elements of the group as reduced loops in the graph. However, some additional structure is added by the edge and vertex groups. We shall use a definition adapted from \cite[p.198]{cohen}, restricted to the case where $X$ has trivial edge groups (and hence trivial monomorphisms $\rho_{e}$).

Let $Y_0$ be a spanning tree in $Y$. Then one can define the fundamental group of $X$ to be 

\begin{align*}
	\pi_1(X) = \frac{\big(\Asterisk_{v\in V(Y)}G_v\big)*F(E(Y))}{N}
\end{align*}

where $N$ is the normal closure of the set $\{e\overline{e}\mid e\in E(Y)\}\cup \{e\in Y_0\}$. This can be simplified into the form $\pi_1(X) \cong G_1*\ldots*G_k*F_r$, where $G_1,\ldots,G_k$ are the nontrivial vertex groups, and $F_r\cong \pi_1(Y)$ is a free group of rank $r\geq 0$. This definition does not depend upon the choice of $Y_0$.

We will be working with graphs of groups whose fundamental group is isomorphic to a particular group $G$. Thus we consider pairs $(X,\phi)$, where $X$ is a graph of groups and \\$\phi:G\to \pi_1(X)$ is an isomorphism. Such a pair is called a \emph{marked graph of groups}, and $\phi$ is called the \emph{marking}.

\subsection*{- The Quotient Graph of Groups}

Given a $G$-tree $T$, we can construct from it a metric graph of groups called the \emph{quotient graph of groups}. A comprehensive method for constructing a quotient graph of groups from an arbitrary connected graph acted upon by $G$ can be found on pages 204-205 of \cite{cohen}. For this paper, we shall restrict his construction to edge-free $G$-trees.

Let $T$ be a $G$-tree. Take the quotient graph $T/G$, and let $p:T\to T/G$ be the projection map, and $Y_0$ a maximal tree of $T/G$. Let $j:Y_0\to T$ be a map such that $p\circ j$ is the identity on $Y_0$ (i.e $j$ is a lift of $Y_0$ to $T$). We call $j(Y_0)$ a \emph{representative tree} for the action.

We then define a graph of groups $X$ on $T/G$ as follows: For any vertex $x$ of $T/G$, we define the vertex group $G_x$ to be $\text{stab}(j(x))$. We take all edge groups, and hence all edge monomorphisms, to be trivial, and edges inherit their lengths from $T$. This completely defines $X$.

The metric on $X$ is given by assigning the length of each edge $e$ of $X$ to be the length of its corresponding edge orbit in $T$.

\begin{thm}\cite[p.210, Thm 26 (iii)]{cohen}\label{thm:fundamental_group_is_G}
	Let $T$ be a $G$-tree, and let $X$ be a quotient graph of groups for $T$. Then the fundamental group of $X$ is isomorphic to $G$.
\end{thm}

This isomorphism gives a marking on $X$, and hence we can think of the quotient graph of groups as a marked graph of groups.

\subsection*{- The Universal Cover}

Conversely, let $(X,\phi)$ be a marked metric graph of groups with \\$G\stackrel{\phi}{\cong} \pi_1(X,v)$. Then we can construct from $X$ a $G$-tree called the \emph{Bass-Serre tree}, or \emph{universal cover} of $X$, denoted by $\tilde{X}$. The process of constructing the universal cover is well-documented in the literature (e.g. \cite[p.205]{cohen}), so we shall not cover it here. 






\begin{defn}
	We say that two marked metric graphs of groups are \emph{equivalent} if their universal covers are equivalent $G$-trees.
\end{defn}

\begin{thm}[Fundamental Theorem of Bass-Serre Theory]
	The process of lifting to the universal cover and the process of descending to a quotient graph of groups are mutually inverse, up to equivalence of the structures involved.
\end{thm}

Let $T\in\mathcal{O}$. We observe that a quotient graph of groups $(X,\phi)$ of $T$ is not unique: while the underlying graph is always $T/G$, the vertex groups depend on our choice of $j$, and hence $X$ is not unique.	Additionally, given a choice of $X$, the marking $\phi:G\to\pi_1(X)$ is not unique.

However, it follows from the fundamental theorem of Bass-Serre Theory that all possible choices of $j$ and $\phi$ give equivalent marked graphs of groups. Thus, for brevity of notation, we shall simply denote a marked graph of groups $(X,\phi)$ by $X$.

\section{Free Factor Systems and the Deformation Space}

Let $G$ be a group, and let $T$ be a $G$-tree.

\begin{defn}
	An element $g\in G$ is said to be \emph{elliptic} (with respect to $T$) if it fixes a point in $T$. If $g$ is not elliptic, we say it is \emph{hyperbolic} (with respect to $T$).
	
	We shall say a subgroup $H$ of $G$ is \emph{elliptic} (with respect to $T$) if there exists a point $x\in T$ such that $x\cdot H = x$.
\end{defn}

\begin{defn}[Free Factor System]\label{defn:freefactorsystem}
	Let $T$ be a minimal, cocompact, edge-free $G$-tree, and let $\mathcal{G}_T$ denote the set of elliptic subgroups for $T$. We say $\mathcal{G}_T$ is a \emph{free factor system} for $G$.
\end{defn}

Note that this is not the usual definition of a free factor system. The usual definition can be found in \cite[p.530-531]{bestvinafeighnhandel}, and we shall refer to it as a \emph{traditional free factor system}:

\begin{defn}[Traditional Free Factor System]
	If $G_1*\ldots*G_k*F_r$ is a free product decomposition for a group $G$, and each $G_i$ is nontrivial, then we say that the collection $\{[G_1],\ldots,[G_k]\}$ of conjugacy classes is a \emph{traditional free factor system}. The empty set $\emptyset$ is the trivial traditional free factor system.
\end{defn}

Corollory \ref{cor:ffs_equivalence} will show that the Fundamental Theorem of Bass-Serre Theory provides a natural way to construct a free factor system from a traditional free factor system, and vice versa, and that these constructions are mutually inverse. In this sense, the two definitions are equivalent.

We have chosen our definition for two reasons. Firstly, the trivial free factor system must be defined separately when using the traditional definition. Secondly, our definition allows us to order our free factor systems by inclusion, and this corresponds to the somewhat more complicated ordering used for the traditional free factor systems, as defined in \cite[p.532]{bestvinafeighnhandel}.

For now, we return to our definition of a free factor system and we observe the following properties:

\begin{lem}
	Free factor systems are closed under conjugation and taking subgroups.
\end{lem}

\begin{lem}
	Let $(T,\cdot),(S,*)$ be equivalent minimal, cocompact, edge-free $G$-trees. Then \\$\mathcal{G}_T=\mathcal{G}_S$
\end{lem}

\begin{lem}\label{lem:elliptics_are_vertex_stabs}
	Let $(T,\cdot)$ be an edge-free $G$-tree. A nontrivial element of $G$ cannot fix more than one point in $T$, and the fixed point will always be a vertex.
\end{lem}

The final lemma tells us that a subgroup of $G$ is elliptic with respect to a minimal, co-compact, edge-free $G$-tree, and hence is an element of the corresponding free factor system, if and only if it is a vertex stabiliser or a subgroup of a vertex stabiliser.

\begin{defn}
	Let $\mathcal{G}$ be a free factor system for $G$. The \emph{deformation space} $\mathcal{O} = \mathcal{O}(G,\mathcal{G})$ is the space of equivalence classes of minimal, cocompact, edge-free $G$-trees $T$ such $\mathcal{G}_T = \mathcal{G}$
\end{defn}

By Lemma \ref{lem:elliptics_are_vertex_stabs}, a subgroup of $G$ is in $\mathcal{G}$ if and only if it is a vertex stabiliser or subgroup of a vertex stabiliser for some (and hence every) $G$-tree in $\mathcal{O}$. Additionally, by general properties of group actions, two vertices lie in the same orbit if and only if they have conjugate stabilisers. Thus we can define a \emph{minimal generating set} for $\mathcal{G}$:

\begin{defn}
	We say a subset of a free factor system $\mathcal{G}$ is a \emph{minimal generating set} for $\mathcal{G}$ if and only if it contains exactly one vertex stabiliser from every orbit of non-free vertices in some (and hence every) $T\in\mathcal{O}$.
\end{defn}

\begin{rem}
	If $\mathcal{G}$ is non-trivial, then this definition is the conventional definition of a minimal generating set under the operations of conjugation and taking subgroups.
	
	If $\mathcal{G} = \{1\}$ - the trivial free factor system - then trees in $\mathcal{O}(G,\mathcal{G})$ do not contain any non-free vertices. Thus the minimal generating set for $\mathcal{G} = \{1\}$ is the empty set $\emptyset$.
\end{rem}

$G$-trees in $\mathcal{O}$ are cocompact, so by Proposition \ref{prop:finite_iff_compact}, $T$ will have a finite number of vertex orbits. Hence a minimal generating set for $\mathcal{G}$ will always be finite.

\begin{thm}\label{thm:factor_basis}
	The following are equivalent:
	\begin{itemize}
		\item[(i)] There exists a minimal, cocompact, edge-free $G$-tree $T$ containing a representative tree $T_0$ in $T$ such that the non-trivial vertex stabilisers in $T_0$ are exactly $G_1,\ldots,G_k$. 
		\item[(ii)] There exists a minimal, cocompact, edge-free $G$-tree $T$ and a quotient graph of groups on $T/G$ whose non-trivial vertex groups are exactly $G_1,\ldots,G_k$.
		\item[(iii)] $G$ can be written as a free product $G=G_1*\ldots*G_k*F_r$, where $F_r$ is a free group of rank $r\geq 0$.
	\end{itemize}
\end{thm}
\begin{proof}
	
	
	
	
	
	$(i)\Leftrightarrow(ii)$ Follows from the Fundamental Theorem of Bass-Serre Theory.
	
	$(ii)\Rightarrow(iii)$ Follows immediately from the definition of the fundamental group of a graph of groups and Theorem \ref{thm:fundamental_group_is_G}.
	
	$(iii)\Rightarrow (ii)$ We must first consider the case where $k=1$ and $r=0$ - that is to say, the trivial free product decomposition. In this case, we take $X$ to be the graph of groups consisting of a single vertex with vertex group $G$.
	
	Otherwise, we take $X$ to be the graph of groups which consists of:
	\begin{itemize}
		\item a rose with central vertex $v_{\infty}$ (with trivial vertex group) and $r$ petals.
		\item for each $i\in\{1,\ldots,k\}$:
		\begin{itemize}
			\item a vertex $v_i$ with associated vertex group $G_i$
			\item an edge $e_i$ between $v_i$ and $v_{\infty}$
		\end{itemize}
		\item trivial edge groups for every edge
	\end{itemize}
	Then $\pi_1(X) = G_1*\ldots*G_k*F_r = G$. Upon lifting to the universal cover, it can be seen that this is a quotient graph of groups for a minimal, cocompact, edge-free $G$-tree.
\end{proof}

We can now demonstrate the correspondence between free factor systems and traditional free factor systems.

Let $\mathcal{G}$ be a free factor system for a group $G$. This means that $\mathcal{G}$ is the set of elliptic subgroups of some $T\in\mathcal{O}(G,\mathcal{G})$. Take a representative tree in $T$, and let $\{G_1,\ldots,G_k\}$ be the nontrivial vertex groups of this representative tree. Then by Theorem \ref{thm:factor_basis}, $G$ can be written as a free product $G=G_1*\ldots*G_k*F_r$. Thus the set $\{[G_1],\ldots,[G_k]\}$ is a traditional free factor system. Since each $[G_i]$ is a conjugacy class, this set does not depend on the choice of representative tree.

Conversely, let $\{[G_1],\ldots,[G_k]\}$ be a traditional free factor system. Then we have \\$G=G_1*\ldots*G_k*F_r$, so by Theorem \ref{thm:factor_basis}, there exists a minimal, cocompact, edge-free $G$-tree $T$ containing a representative tree $T_0$ in $T$ such that the non-trivial vertex stabilisers in $T_0$ are exactly $G_1,\ldots,G_k$. These vertex groups give a minimal generating set for a free factor system $\mathcal{G} = \mathcal{G}_T$.

\begin{cor}\label{cor:ffs_equivalence}
	The process of constructing a free factor system from a traditional free factor system, and the process of constructing a traditional free factor system from a free factor system, are mutually inverse. In this sense, the two definitions are equivalent.
\end{cor}
\begin{proof}	
	Follows from Theorem \ref{thm:factor_basis}.
\end{proof}



It follows from Corollary \ref{cor:ffs_equivalence} that, given a free product $G=G_1*\ldots*G_k*F_r$, we can construct a free factor system $\mathcal{G}$, and hence a corresponding deformation space $\mathcal{O}(G,\mathcal{G})$. This space can then be used to study the automorphisms of the free product.

\section{Automorphisms}

For the duration of this section, let $\mathcal{G}$ denote a free factor system for a group $G$, and let $\mathcal{O} = \mathcal{O}(G,\mathcal{G})$.

\subsection*{- Acting on the Deformation Space}

\begin{nota}
	The outer automorphism group of $G$ is defined as $\text{Out}(G) := \faktor{\text{Aut}(G)}{\text{Inn}(G)}$; elements of $\text{Out}(G)$ are equivalence classes of automorphisms, where two automorphisms are equivalent if they differ by an inner automorphism.
	
	When we write $\alpha\in\text{Out}(G)$, we mean that $\alpha$ is an automorphism in $\text{Aut}(G)$ representing an equivalence class in $\text{Out}(G)$.
	
	In this paper, the automorphisms of $G$ will act on $G$ on the right.
\end{nota}

\begin{defn}\label{def:invariantautos}
	Let $\alpha\in\text{Aut}(G)$, and let $\mathcal{G}\a = \{(H)\a\mid H\in\mathcal{G}\}$. We say that $\mathcal{G}$ is \emph{$\alpha$-invariant} if $ \mathcal{G}\a = \mathcal{G}$.
\end{defn}

Free factor systems are closed under conjugation by elements of $G$, hence $\a$-invariance depends only on the outer automorphism class of $\a$. Thus we can make a similar definition for $\text{Out}(G)$:

\begin{defn}\label{def:invariantouterautos}
	Let $\alpha\in\text{Out}(G)$, and let $\mathcal{G}\a = \{(H)\a\mid H\in\mathcal{G}\}$. We say that $\mathcal{G}$ is \emph{$\alpha$-invariant} if $\mathcal{G}\a = \mathcal{G}$.
	
	The set of $\mathcal{G}$-invariant outer automorphisms of $G$ forms a group, which we shall denote $\text{Out}(G,\mathcal{G})$.
\end{defn}

The group $\text{Out}(G,\mathcal{G})$ admits a left action on the deformation space $\mathcal{O} = \mathcal{O}(G,\mathcal{G})$: Let $(T,d_T,\cdot)\in\mathcal{O}$, let $\alpha\in\text{Out}(G,\mathcal{G})$. Then $\alpha (T,d_T,\cdot) := (T,d_T,\cdot_{\alpha})$, the $G$-tree with the same underlying simplicial tree and metric, but with `twisted' action given by $x\cdot_{\alpha} g = x\cdot (g)\a$ for all $x\in T$.

Observe the following:

\begin{lem}\label{lem:alpha_preserves_orbits}
	Let $T\in\mathcal{O}$ be a $G$-tree, and let $\alpha\in \text{Out}(G,\mathcal{G})$. Then the $G$-orbits of $T$ are the same as those of $\a T$. That is, for all $x\in T$, $x\cdot G = x \cdot_\alpha G$.
\end{lem}
\begin{proof}
	$x\cdot G = \lbrace x\cdot g \mid g\in G \rbrace = \lbrace x\cdot (g)\a \mid g\in G \rbrace = x\cdot_\alpha G$
\end{proof}

Thus we see that the `twists' $\alpha$ applies to the action only occur within each orbit. This means that if we are working with both $(T,\cdot)$ and $(T,\cdot_\alpha)$, we are able to simply refer to `a $G$-orbit of $T$' without having to state which action is being used.

\begin{defn}\label{defn:irreducible}
	We can partially order the set of all free factor systems of $G$ by inclusion. Let $\mathcal{G}$ be a proper, $\a$-invariant free factor system. We say $\a\in \text{Out}(G,\mathcal{G})$ is \emph{$\mathcal{G}$-irreducible}, or \emph{irreducible with respect to $\mathcal{G}$}, if $\mathcal{G}$ is a maximal, proper $\a$-invariant free-factor system.
	
	Otherwise, we say $\a$ is \emph{reducible} with respect to $\mathcal{G}$.
\end{defn}

\subsection*{- Topological Representatives}

\begin{defn}\cite[p.16]{francavigliamartino2018a}
	Let $T,S\in \mathcal{O}(G,\mathcal{G})$. An \emph{$\mathcal{O}$-map} $f:T\to S$ is a $G$-equivariant, Lipschitz continuous function. The Lipschitz constant of $f$ is denoted $\text{Lip}(f)$.
\end{defn}

Note that an $\mathcal{O}$-map does not have to send vertices to vertices, and hence does not need to be a graph morphism.

\begin{defn}\cite[p.16]{francavigliamartino2018a}
	We say an $\mathcal{O}$-map $f:T\to S$ is \emph{straight} if it has constant speed on edges - that is, for each edge $e$ in $T$, there exists a non-negative number $\l_e(f)$ such that for any $a,b\in e$ we have $d_T(f(a), f(b)) = \l_e(f)d_S(a, b)$.
\end{defn}

\begin{defn}\label{def:topological_representative}
	Let $\alpha\in\text{Out}(G,\mathcal{G})$, and let $T\in\mathcal{O}$ be a $G$-tree. Then a map $f:T\to \alpha T$ is said to \emph{topologically represent} $\alpha$ if it is a straight $\mathcal{O}$-map.
\end{defn}

The authors of \cite[p. 16]{francavigliamartino2018a} make the following remark:

\begin{rem}\label{rem:Omaps_exist}
	Any two trees $T,S\in\mathcal{O}$ have an $\mathcal{O}$-map between them. Furthermore, any $\mathcal{O}$-map $f:T\to S$ can be uniquely 'straightened' - that is to say, there exists a unique straight $\mathcal{O}$-map $\text{Str}(f):T\to S$, such that $\text{Str}\big(f\big)(v) = f(v)$ for every vertex $v\in T$. We have \\$\text{Lip}(\text{Str}(f))\leq \text{Lip}(f)$.
\end{rem}

From this remark it follows that $\forall \a\in\text{Out}(G,\mathcal{G})$, $\forall T\in\mathcal{O}$, there exists a topological representative $f:T\to \a T$.


\begin{defn}\label{defn:hyperbolic_components}
	Let $F$ be a subforest of some $T\in\mathcal{O}$, and let $A$ be a component of $F$. We define the \emph{stabiliser} of $A$ to be the set $\text{stab}(A) = \{g\in G \mid A\cdot g = A\}$ - that is, we are taking the setwise stabiliser, not the pointwise stabiliser.
	
	We say that $F$ is $\mathcal{G}$-elliptic if, for every component $A$ of $F$, $\text{stab}(A)\in \mathcal{G}$. Otherwise we say that $F$ is $\mathcal{G}$-hyperbolic.
\end{defn}

Using topological representatives, we can construct a test for the reducibility of an automorphism:

\begin{thm}\label{thm:reducibility_test}
	Let $\mathcal{G}$ be a proper free factor system for a group $G$, let  $\alpha\in\text{Out}(G,\mathcal{G})$, and let $T\in\mathcal{O}$.
	
	Suppose that $\alpha$ can be topologically  represented by a $G$-equivariant simplicial map $f:T\to \alpha T$, and there exists a proper $f$-invariant, $G$-invariant, $\mathcal{G}$-hyperbolic subforest of $T$. Then $\alpha$ is reducible with respect to $\mathcal{G}$.
\end{thm}

\begin{proof}
	
	We shall prove that $\alpha$ is reducible by constructing a new cocompact, minimal, edge-free $G$-tree $S$, from which we shall retrieve another proper $\alpha$-invariant free factor system $\mathcal{H}$ such that $\mathcal{G}\subset \mathcal{H}$.
	
	Let $F$ denote the subforest of $T$ described above, together with all the remaining vertices of $T$. This extended subforest is still proper, $f$-invariant, and $G$-invariant. (In addition, recall that we defined subforests such that they respect the simplicial structures of our $G$-trees; therefore the complement of $F$ is a set of edges.)
	
	We obtain $S$ from $T$ by collapsing each component $A$ of $F$ to a point $p_A$.  Edges which were not collapsed inherit their lengths from $T$, giving us a metric on $S$. Since $F$ is $G$-invariant, this collapse induces a minimal, isometric action of $G$ on $S$. Thus $S$ is a $G$-tree.
	
	Furthermore, if we declare the vertex set to be the set of $p_A$, we induce a new simplicial structure on $S$. (This is a well-defined vertex set, and the simplicial structure given by this vertex set is exactly the same as the usual simplicial structure we give to all $G$-trees, as defined in Section \ref{sec:1.2}).
	
	$S$ inherits cocompactness and edge-freeness from $T$. Hence, by definition, the set $\mathcal{H}$ of elliptic subgroups for $S$ is a free factor system for $G$.
	
	Let $v\in T$ be a vertex. Since $F$ contains every vertex of $T$, $v$ must lie in some component $A$ of $F$. Therefore, since $F$ is $G$-invariant, all of $A$ must be fixed (setwise, not necessarily pointwise) by $\text{stab}(v)$. Hence $\text{stab}(v) \leq \text{stab}(p_{A})$. This holds for all $v$, which is enough to tell us that $\mathcal{G}\subseteq \mathcal{H}$.
	
	Some component of $F$ has $\mathcal{G}$-hyperbolic stabiliser. This means that $\text{stab}(p_{A})\notin \mathcal{G}$ but $\text{stab}(p_{A})\in \mathcal{H}$. Thus $\mathcal{G}\subset \mathcal{H}$.
	
	Suppose that $G\in\mathcal{H}$ - that is to say, a vertex of $S$ is stabilised by $G$. Then the component of $F$ corresponding to this vertex is a $G$-invariant subtree of $T$, contradicting the minimality of $T$. Hence $G\notin \mathcal{H}$, so $\mathcal{H}$ is a proper free factor system.
	
	Finally, we must show that $\mathcal{H}$ is $\alpha$-invariant:
	
	Let $p_A$ be a vertex of $S$. We first want to show that $(\text{stab}(p_A))\a$ lies in $\mathcal{H}$ - that is to say, it fixes a point in $(S,*)$. $F$ is $f$-invariant, therefore $f(A)$ is a component of $F$ and $p_{f(A)}$ is a vertex of $S$. Furthermore, $\forall (g)\a\in (\text{stab}(v_A))\a$, $p_{f(A)}*(g)\a = p_{f(A)\cdot (g)\a} = p_{f(A\cdot g)} = p_{f(A)}$, and hence $(\text{stab}(p_A))\a\in\mathcal{H}$. This holds for all $p_A\in S$. This tells us that $\mathcal{H}\a\subseteq \mathcal{H}$, which in turn is enough to show that $\mathcal{H} = \mathcal{H}\a$.
	
	To summarize, $\mathcal{H}$ is a proper, $\alpha$-invariant  free-factor system for $G$, and $\mathcal{G}\subset \mathcal{H}$. Hence, by Definition \ref{defn:irreducible}, $\alpha$ is reducible with respect to $\mathcal{G}$.
\end{proof}

\subsection*{- Isometric topological representatives}

Our main results will make use of \emph{isometric} topological representatives, which allow us to make some additional observations:

\begin{rem}
	Recall that $\mathcal{O}$ is a space of equivalence classes of $G$-trees, where two trees are equivalent if there exists an equivariant isometry between them. Topological representatives are equivariant; therefore, if an isometric topological representative $f:T\to \a T$ exists, the two $G$-trees $T$ and $\a T$ are representing the same point in $\mathcal{O}$.
\end{rem}

\begin{prop}\label{prop:isometries_are_graph_autos}
	Let $f:T\to \a T$ be a topological representative for some $T\in \mathcal{O}(G,\mathcal{G})$, for some $\a\in \text{Out}(G,\mathcal{G})$. If $f$ is an isometry, then it is also a graph automorphism.
\end{prop}
\begin{proof}
	It is sufficient to show that $f(v)$ is a vertex if and only if $v$ is a vertex.
	$f$ is an isometry 2- in particular it is bijective - therefore $f(v)$ is a branch point if and only if $v$ is a branch point. The only vertices which remain are the degree 2 vertices. Recall that these were introduced as the midpoints of inverted edges, hence they all have stabiliser of order 2. $f$ is equivariant, therefore $\text{stab}(v)$ is order 2 if and only if $\text{stab}(f(v))$ is order 2. Thus the set of degree 2 vertices is also preserved, and $f$ is a graph automorphism.
\end{proof}

\begin{defn}
	Let $Y$ be a metric graph. The \emph{volume} of $Y$, denoted $\text{Vol}(Y)$ is defined to be the sum of the lengths of the edges of $Y$.
	
	Let $T\in\mathcal{O}$. The \emph{covolume} of $T$, denoted $\text{Covol}(T)$, is defined to be the volume of the graph $T/G$.
\end{defn}

\begin{prop}\label{prop:isom_iff_Lip1}
	Let $f:T\to \a T$ be a topological representative for some $T\in \mathcal{O}(G,\mathcal{G})$, for some $\a\in \text{Out}(G,\mathcal{G})$. Then $\text{Lip}(f)=1$ if and only if $f$ is an isometry.
\end{prop}
\begin{proof}
	If $f$ is an isometry, then $\text{Lip}(f)=1$ follows immediately. It remains to prove the converse.
	
	Let $D$ be a subforest of $T$ consisting of exactly one edge from each orbit. Then \\$\text{Covol}(T) = \text{Vol}(D)$. Without loss of generality, we may assume that $\text{Covol}(T)=1$. Since $T$ and $\a T$ have the same metric, this means that $\text{Covol}(\a T)=1$
	
	Since $f$ is equivariant, $f(D)$ contains a fundamental domain for $f(T)$. Since $\a T$ does not contain any proper invariant subtrees, we must have $f(T) = \a T$, hence $f(D)$ contains a fundamental domain for $\a T$. It follows that $\text{Vol}(f(D))\geq 1$. In addition, 
	\begin{align*}
		\text{Vol}(f(D)) &\leq \sum_{\text{edges } e\in D} L(f(e)) &&\text{(*)}\\
		&\leq  \sum_{\text{edges } e\in D} L(e) &&\text{(as $\text{Lip}(f)=1$)}\\
		&= \text{Vol}(D)\\
		&=\text{Covol}(T)\\
		&=1
	\end{align*}
	
	We split into two cases:
	
	\textbf{Case 1: $f$ is not locally injective} This is equivalent to saying that $f$ `folds' a pair of edges - that is, there is a vertex $v$, neighbourhoods $U_1, U_2$ of $v$, and edges $e_1, e_2$ incident to $v$ such that $f(e_1\cap U_1) = f(e_2\cap U_2)$. (The neighbourhoods are required because $f$ may not fold the entirety of the edges, only the initial segments. Since $f$ may stretch these segments, the neighbourhoods are not the same size in general).
	
	The two edges can only be folded if they lie in different orbits; observe that if $e_1\cdot g = e_2$, then $f(e_1\cap U_1)$ must be fixed by $\a(g)$, contradicting edge freeness. Therefore we are free to choose $D$ such that it contains a pair of folded edges.
	
	If $e_1,e_2$ are a pair of folded edges in $D$, then the volume of their image under $f$ will be strictly less then the sum of their original lengths. This means that (*) is a strict inequality, so $\text{Vol}(f(D)) < 1$. This contradicts $\text{Vol}(f(D))\geq 1$, hence Case 1 cannot occur.
	
	\textbf{Case 2: $f$ is locally injective.} Then (*) is an equality, so $\text{Vol}(f(D)) = 1$. This is enough to tell us that $L(f(e)) = L(e)$ for every edge $e$ of $D$, and hence all the edges of $T$; thus $f$ is an isometry on every edge of $T$. This, combined with local injectivity, means that $f$ is an isometry on all of $T$.
	
\end{proof}

\section{Distance on $\mathcal{O}$}

\subsection*{- Stretching Factors}

\begin{defn}
	Let $g\in G$, and let $T\in\mathcal{O}(G,\mathcal{G})$. The \emph{translation length} of $g$ in $T$, denoted $l_T(g)$, is defined as
	\begin{center}
		$l_T(g) = \displaystyle\inf_{x\in T} \{d_T(x, x\cdot g)\}$
	\end{center}
\end{defn}

\begin{rem}	
	This infimum is in fact a minimum, and is obtained for some $x$. If $g$ is elliptic then this is observed to be true from the definition of an elliptic element, and we have $l_T(g) = 0$.
	
	If $g$ is hyperbolic then the translation length will be non-zero, and the set of elements realising this length will form a line through $T$ called the \emph{hyperbolic axis of $g$}. Points on the axis will be translated along the axis by $l_T(g)$.
\end{rem}

\begin{rem}
	An equivalence class of $G$-trees in $\mathcal{O}$ is uniquely determined by its translation length function \cite{cullermorgan} - thus one can think of $\mathcal{O}$ as being embedded in the space $\R^G$.
\end{rem}

\begin{defn}
	Let $\text{Hyp}(\mathcal{G})$ denote the set of elements of $G$ which do not lie in any subgroup of $\mathcal{G}$. (In other words, $\text{Hyp}(\mathcal{G})$ is the set of elements which are hyperbolic with respect to some, and hence all, $G$-trees in $\mathcal{O}(G,\mathcal{G})$).
\end{defn}

\begin{defn}\cite[p.8]{francavigliamartino2015}\label{def:stretching_factors}
	Let $T,S\in \mathcal{O}(G,\mathcal{G})$. Then we define the \emph{left and right stretching factor} from $T$ to $S$ as
	\begin{align*}
		\Lambda_L(T,S) := \sup_{g\in \text{Hyp}(\mathcal{G})}\dfrac{l_T(g)}{l_S(g)} \hspace{35pt} \Lambda_R(T,S) := \sup_{g\in \text{Hyp}(\mathcal{G})}\dfrac{l_S(g)}{l_T(g)} = \Lambda_L(S,T)
	\end{align*}
	respectively. We also define the \emph{symmetric stretching factor} from $T$ to $S$ to be
	\begin{align*}
		\Lambda(T,S) := \Lambda_L(T,S)\Lambda_R(T,S)
	\end{align*}
\end{defn}

The next Theorem follows from \cite[Corollary 6.8, p.18 and Theorem 6.11, p.19]{francavigliamartino2015}.

\begin{thm}\label{thm:lip_realises_sf}
	Let $T,S\in\mathcal{O}$. Then there exists a Lipschitz continuous map $f:T\to S$ such that $\text{Lip}(f) = \Lambda_R(T,S)$
\end{thm}

\subsection*{- The Displacement of an Automorphism}

\begin{defn}
	Let $\a\in \text{Out}(G,\mathcal{G})$. Then we define the \emph{displacement} of $\a$ to be
	\begin{align*}
		\l_\a :=\inf_{T\in\mathcal{O}}\Lambda_R(T,\a T)
	\end{align*}
\end{defn}

\begin{thm}\cite[p. 25]{francavigliamartino2015}\label{thm:irreduciblerealisesdisplacement}
	For any $\mathcal{G}$-irreducible $\a\in \text{Out}(G,\mathcal{G})$, the displacement of $\a$ is a minimum and obtained for some $T\in\mathcal{O}$.
\end{thm}

\begin{defn}
	For any $\a\in\text{Out}(G,\mathcal{G})$, we define
	\begin{align*}
		\text{Min}(\a) = \{T\in\mathcal{O}\mid \Lambda_R(T,\a T) = \lambda_\a \}
	\end{align*}
	That is to say, $\text{Min}(\a)$ is the set of all $T$ which realise the above infimum.
\end{defn}

\begin{thm}\label{thm:isom_top_rep_exists}
	Let $\a\in \text{Out}(G,\mathcal{G})$ be a $\mathcal{G}$-irreducible, displacement 1 automorphism. Then for all $T\in \text{Min}(\a)$, there exists an isometric topological representative for $\a$ on $T$.
\end{thm}
\begin{proof}	
	Let $T\in \text{Min}(\a)$. Then by definition of the minimally displaced set, $\Lambda_R(T,\a T) = \lambda_\a = 1$, and hence by Theorem \ref{thm:lip_realises_sf} there exists a Lipschitz continuous map $f:T\to \a T$ with $\text{Lip}(f) = 1$. Therefore, by Proposition \ref{prop:isom_iff_Lip1}, $f$ is an isometric topological representative for $\a$.
\end{proof}

\begin{cor}\label{cor:minsetisfixset}
	Let $\a\in \text{Out}(G,\mathcal{G})$ be a $\mathcal{G}$-irreducible, displacement 1 automorphism. Then $\text{Min}(\a) = \text{Fix}(\a)$.
\end{cor}
\begin{proof}
	Let $T\in \text{Fix}(\a)$. Then $T$ and $\a T$ are equivalent $G$-trees, so $\Lambda(T,\a T) = 1 = \lambda_\a$. Thus $\text{Fix}(\a) \subseteq \text{Min}(\a)$.
	
	Conversely, let $T\in \text{Min}(\a)$. By Theorem \ref{thm:isom_top_rep_exists}, there exists an equivariant isometry from $T$ to $\a T$.  Points in $\mathcal{O}$ are equivalence classes of $G$-trees under equivariant isometry, hence $T$ and $\a T$ represent the same point in $\mathcal{O}$. Thus $\text{Min}(\a)\subseteq \text{Fix}(\a)$.
\end{proof}

\section{Secondary Theorem}

We extend some of our terminology for $G$-trees to graphs of groups:

\begin{defn}
	Let $T \in \mathcal{O}(G,\mathcal{G})$, and let $X$ be a quotient graph of groups on $T/G$. We shall say that a vertex of $T/G$ is \emph{free} if it has trivial vertex group in $X$. Otherwise, it is \emph{non-free}. Note that this definition does not depend on our choice of $X$.
\end{defn}

\begin{defn}
	Let $T\in \mathcal{O}(G,\mathcal{G})$, and let $X$ be a quotient graph of groups on $T/G$. We say that a subgraph-of-groups of $X$ is $\mathcal{G}$-elliptic if and only if the fundamental group of all its components lies in $\mathcal{G}$. Otherwise, we say it is $\mathcal{G}$-hyperbolic.
	
	Similarly, we say that a subgraph of $T/G$ is $\mathcal{G}$-elliptic/hyperbolic if the corresponding subgraph-of-groups of $X$ is $\mathcal{G}$-elliptic/hyperbolic. Observe that this definition does not depend on the choice of marking on $X$.
\end{defn}

It follows from the Fundamental Theorem of Bass-Serre Theory that a $G$-invariant subforest of $T$ is $\mathcal{G}$-elliptic if and only if it collapses to a $\mathcal{G}$-elliptic subgraph of $T/G$.

We also observe that, by the definition of the fundamental group, a subgraph of $T/G$ will be $G$-elliptic if and only if each component is a tree containing at most one non-free vertex.

Let $f:T\to \alpha T$ be a topological representative. Topological representatives are equivariant, hence $f$ induces a well-defined map $\varphi :T/G\mapsto T/G$. (Observe that since orbits in $T$ and $\a T$ are the same, $T/G = \a T/G$).

Suppose that $f$ is an isometry. Then by Proposition \ref{prop:isometries_are_graph_autos}, $f$ is a graph automorphism. It follows that $\varphi$ is also an isometric graph automorphism - in particular, it is invertible. Thus we can think of the cyclic group $\langle \varphi \rangle$ as acting on $T/G$.

$\varphi$ can be used in an equivalent form of the reducibility test (Theorem \ref{thm:reducibility_test}), this time using the quotient graph:

\begin{thm}\label{thm:reducibility_test_2}
	Let $\mathcal{G}$ be a proper free factor system for a group $G$, let  $\alpha\in\text{Out}(G,\mathcal{G})$, and let $T\in\mathcal{O}$.
	
	Suppose that $\alpha$ can be topologically  represented by a $G$-equivariant simplicial map $f:T\to \alpha T$, and there exists a proper $\varphi$-invariant, $\mathcal{G}$-hyperbolic subgraph of $T/G$. Then $\alpha$ is reducible with respect to $\mathcal{G}$.
\end{thm}

This form of the reducibility test eliminates the need to check for $G$-invariance. Note that a subgraph of $T/G$ is $\varphi$-invariant if and only if it is invariant under the action of $\langle \varphi \rangle$.

\begin{defn}
	We say a graph $Y$ is a \emph{star} if it is a tree and there exists a vertex $w$ which is incident to every edge of $Y$.
\end{defn}

\begin{lem}\label{lem:degree_1_and_2_v.groups}
	Let $T\in \mathcal{O}$, and let $X$ be a quotient graph of groups for $T$. Then all the vertices of degree 1 or 2 in $X$ will have non-trivial vertex groups.
\end{lem}


The proof of this lemma follows directly from Propositions \ref{prop:no_degree_1} and \ref{prop:no_free_degree_2}.

\begin{thm}\label{thm:permuting_edge_orbits}
	Let $\mathcal{G}$ be a free factor system for a group $G$, and let $\a\in\text{Out}(G,\mathcal{G})$ be an irreducible automorphism with $\l_\a =1$. Let $T\in\text{Min}(\a)$, and let $f:T\to \a T$ be an equivariant isometry ($f$ exists by Theorem \ref{thm:isom_top_rep_exists}).
	
	Then $f$ cyclically permutes the $G$-orbits of edges in $T$.
\end{thm}
\begin{proof}
	$f$ induces a map $\varphi$ on $T/G$, and the cyclic group $\langle \varphi \rangle$ acts on $T/G$. The theorem statement is equivalent to saying that $\varphi$ cyclically permutes the edges of $T/G$.
	
	Let $e$ be an edge in $T/G$. We define two subgraphs of $T/G$:
	\begin{itemize}
		\item Let $A$ be the subgraph of $T$ with $E(A) = e\cdot\langle \varphi \rangle$ and with $V(A)$ equal to the set of vertices incident to $E(A)$.
		\item Let $B$ be the subgraph of $T$ with edge set $E(B) = E(T)-e\cdot\langle \varphi \rangle$ and with $V(B)$ equal to the set of vertices incident to $E(B)$.
	\end{itemize}
	Observe that these are both $\langle \varphi \rangle$-invariant.
	
	Suppose that $\varphi$ does not cyclically permute the edges of $T/G$. This means that $A$ and $B$ are both proper subgraphs of $T/G$. We shall show that at least one of the two subgraphs is $\mathcal{G}$-hyperbolic. This will mean that $\a$ is reducible by Theorem \ref{thm:reducibility_test_2}, giving us a contradiction.
	
	In more detail, we assume that $A$ is $\mathcal{G}$-elliptic. Then $A$ is a forest such that every component contains at most one non-free vertex. If $B$ is not a forest, then $B$ is immediately $\mathcal{G}$-hyperbolic, so assume that it is a forest. We shall show that some component of $B$ contains at least 2 non-free vertices, and hence $B$ is $\mathcal{G}$-hyperbolic.
	
	To begin we note that by Lemma \ref{lem:degree_1_and_2_v.groups}, $X$ does not contain any free vertices of degree 1 or 2. Thus we make the following claim:
	
	\begin{center}		
		\textbf{Claim (i):} Let $v$ be a free vertex of $T/G$. Then $\text{deg}_{T/G}(v) = \text{deg}_A(v) + \text{deg}_{B}(v)\geq 3$. It follows that, if $\text{deg}_A(v)=1$ or $2$ \emph{or} $\text{deg}_B(v)=1$ or $2$, then $v\in A\cap B$. Additionally, if $\text{deg}_A(v)=1$, then $\text{deg}_{B}(v)\geq 2$, and if $\text{deg}_B(v)=1$, then $\text{deg}_{A}(v)\geq 2$
	\end{center}
	
	Now, $\langle \varphi \rangle$ acts via isometries, and since $A$ is the $\langle \varphi \rangle$-orbit of a single edge, $\langle \varphi \rangle$ acts transitively on the components of $A$. Hence the components of $A$ are all isometric to each other, and we can divide $A$ into two cases:
	
	\textbf{Case 1: Each component of $A$ is a single edge}
	
	$B$ is a finite forest, therefore each component of $B$ has at least two vertices of $B$-degree 1. By Claim (i), if any of these vertices are free, then they must lie in $A\cap B$ and have $A$-degree of at least 2. However, all vertices in $A$ have $A$-degree 1. Hence the $B$-degree 1 vertices are all non-free, and $B$ is $\mathcal{G}$-hyperbolic.
	
	\textbf{Case 2: Each component of $A$ contains more than one edge}
	
	By definition, $A$ contains at most two $\langle \varphi \rangle$-orbits of vertices. Since $A$ is a finite forest, some vertices in $V(A)$ will have $A$-degree 1, and since each component of $A$ contains more than one edge, some vertices in $V(A)$ will have $A$-degree strictly greater than 1. $\langle \varphi \rangle$ acts via graph automorphisms (by Proposition \ref{prop:isometries_are_graph_autos}), and $A$ is $\langle \varphi \rangle$-invariant, therefore vertices in the same $\langle \varphi \rangle$-orbit will have the same $A$-degree. Thus $A$ contains exactly two $\langle \varphi \rangle$-orbits of vertices.
	
	Furthermore, $E(A)$ is the $\langle \varphi \rangle$-orbit of a single edge, so the incident vertices of this edge are representatives for our two vertex orbits. Thus every edge in $A$ must have exactly one incident vertex with $A$-degree 1, and hence $A$ is in fact a disjoint union of stars. We shall refer to the $A$-degree 1 vertices as the \emph{spoke} vertices. The remaining vertices, at the centre of each star, shall be called the \emph{hub} vertices.
	
	By the equivariance of $f$, $\langle \varphi \rangle$ sends free vertices to free vertices, and non-free vertices to non-free vertices. The spoke vertices all lie in the same $\langle \varphi \rangle$-orbit, and each component of $A$ contains at least 2 spoke vertices. Therefore, since $A$ is $\mathcal{G}$-elliptic, the spoke vertices must all be free. Thus, by Claim (i), they must lie in $B$, and have $B$-degree at least 2. (*)
	
	$\langle \varphi \rangle$ acts transitively on the spoke vertices. Therefore $\langle \varphi \rangle$ acts transitively on the components of $B$ which contain the spoke vertices, and hence these components are all isometric to each other. We shall write $B'$ to denote the subforest of $B$ consisting of these components. (Observe that, for any vertex $v\in B'$, $\text{deg}_B(v) = \text{deg}_{B'}(v)$).
	
	We divide into two cases once again:
	
	\textbf{Subcase 1: Each component of $B'$ contains exactly one spoke vertex}
	
	Let $s$ be the number of spoke vertices, and let $l$ be the number of components of $A$. Then $s\geq 2l$.
	
	$B'$ has exactly $s$ components. These components are finite trees, so they will each contain at least 2 vertices of $B'$-degree 1. Thus $B'$ has at least $2s$ vertices with $B'$-degree 1. By Claim (i), any of these vertices which are free must lie in $A$ and have $A$-degree at least 2. However, the only vertices with $A$-degree at least 2 are the $l$ hub vertices. This leaves at least $2s-l$ vertices in $B'$ which must therefore be non-free. $2s-l>s = \text{number of components of } B'$, therefore some component of $B'$ must contain two or more of these non-free vertices. Therefore $B'$, and hence $B$, is $\mathcal{G}$-hyperbolic.
	
	\textbf{Subcase 2: Each component of $B'$ contains more than one spoke vertex}
	
	We divide $B'$ into two subforests, $C$ and $D$:
	\begin{itemize}
		\item Let $B'_1,\ldots, B'_n$ be the components of $B'$. For each $i=1,\ldots,n$, let $C_i$ be the unique minimal subtree of $B'_i$ which contains all the spoke vertices in $B'_i$. Let $C:=\bigcup_{i=1}^n C_i$. Since $B'$ and the set of spoke vertices are $\langle \varphi \rangle$-invariant, $C$ is also $\langle \varphi \rangle$-invariant.
		
		\item Define $D$ to be the subforest of $B'$ consisting of the edge set $E(B') - E(C)$, together with all vertices incident to this edge set. Since $B'$ and $C$ are $\langle \varphi \rangle$-invariant, $D$ is also $\langle \varphi \rangle$-invariant.
	\end{itemize}
	
	By minimality of the $C_i$'s, at least one spoke vertex has $C$-degree 1. Since $C$ is $\langle \varphi \rangle$-invariant and the spoke vertices lie in the same $\langle \varphi \rangle$-orbit, this implies that all the spoke vertices have $C$-degree 1. However, recalling that the spoke vertices have $B$-degree at least 2 (see (*)), this tells us that the spoke vertices are all incident to an edge in $D$ (and hence the spoke vertices themselves are all in $D$).
	
	\begin{center}		
		\textbf{Claim (ii):} A component of $D$ cannot contain more than one point in $C\cap D$.
	\end{center}
	
	\emph{Proof of Claim (ii).} Let $v,w\in C\cap D$, and suppose that $v$ and $w$ lie in the same component of $D$. Then they lie in the same component of $B'$, and hence the same component of $C$. Therefore there exists a unique reduced path $\gamma_D$ from $v$ to $w$ in $D$ , and a unique reduced path $\gamma_C$ from $v$ to $w$ in $C$.
	
	However, $B'$ is a forest, therefore $\gamma_D = \gamma_C$. By definition of $D$, there are no edges in $C\cap D$. Hence both paths are trivial, and $v=w$. This ends the proof of Claim (ii). \qed
	
	In particular, Claim (ii) implies that each spoke vertex lies in a unique component of $D$. Let $D'$ be defined as the subforest of $D$ consisting only of the components which contain spoke vertices. Then, if we let $s$ be the number of spoke vertices, $D'$ will have $s$ components. (Additionally, for any $v\in D'$, $\text{deg}_D(v) = \text{deg}_{D'}(v)$).
	
	Each component of $D'$ will contain at least two vertices of $D'$-degree 1. By Claim (ii), at least one of these will not lie in $C\cap D$, and hence it will also have $B'$-degree 1. Furthermore, since $\langle \varphi \rangle$ acts transitively on the spoke vertices, it will act transitively on the components of $D'$. Thus there exists a $\langle \varphi \rangle$-orbit of at least $s$ vertices with $B'$-degree 1; at least one in each component of $D'$.
	
	By Claim (i), if this orbit of vertices is free, then it must lie in $A$ and have $A$-degree at least 2. However, the only $\langle \varphi \rangle$-orbit of vertices with $A$-degree at least 2 are the $l$ hub vertices. $s\geq 2l$, therefore these cannot be the same orbit. Hence our orbit of $B'$-degree 1 vertices must be non-free.
	
	Each component of $B'$ contains more than one spoke vertex. Therefore each component of $B'$ contains more than one component of $D'$, and hence more than one of our non-free vertices. Thus $B'$ is $\mathcal{G}$-hyperbolic.
\end{proof}

\section{Main Theorem}

For the duration of this section, let $\mathcal{G}$ be a free factor system for a group $G$, and let \\$\mathcal{O} = \mathcal{O}(G,\mathcal{G})$.

\begin{defn}
	The \emph{covolume 1 slice of $\mathcal{O}$}, denoted $\mathcal{O}_1$, is defined to be the subspace of covolume 1 trees in $\mathcal{O}$.
\end{defn}



\begin{defn}
	Let $\a\in\text{Out}(G,\mathcal{G})$. In a similar manner to $\mathcal{O}$, we define the \emph{minimally displaced set} in $\mathcal{O}_1$ to be $\text{Min}_1(\a) = \{T\in\mathcal{O}_1\mid \Lambda_R(T,\a T) = \lambda_\a \}$ and we define \\$\text{Fix}_1(\a) = \{T\in\mathcal{O}_1\mid T \sim \a T \}$.
\end{defn}

\begin{rem}
	For $T\in \mathcal{O}$ and $\mu>0$, let us write $\mu T$ to denote the $G$-tree $(T, \mu d_T, \cdot)$. One can show that for all $T, S\in \mathcal{O}$, $\Lambda_R(T,S) = \Lambda_R(\mu T,\mu S)$ - that is to say, stretching factors are invariant under rescaling the volume of both $G$-trees. Additionally, $T$ and $\a T$ have the same volume for all $T \in \mathcal{O}$, so by rescaling one, we rescale the other. Thus we observe the following:
	\begin{align*}
		\text{Min}(\a) = \{\mu T\in\mathcal{O} \mid T\in \text{Min}_1(\a), \mu > 0 \}\\
		\text{Fix}(\a) = \{\mu T \in\mathcal{O} \mid T\in \text{Min}_1(\a), \mu > 0 \}
	\end{align*}
	It then follows from Corollary \ref{cor:minsetisfixset} that $\text{Min}_1(\a) = \text{Fix}_1(\a)$.
\end{rem}

Let $T\in \mathcal{O}_1$. The metric on $T$ can be completely described by the length of one edge from each $G$-orbit - or equivalently, the lengths of the edges of $T/G$. Hence, if there are $n$ edge orbits with lengths $x_1,\ldots,x_n$, then the open simplex $\{(x_1,\ldots,x_n) \mid x_1+\ldots x_n = 1, x_i> 0 \hspace{7pt} \forall i\}$ describes the set of all possible metrics on $T$. Repeating this for every tree in $\mathcal{O}_1$ allows us to think of $\mathcal{O}_1$ as a union of open simplices.

(Equivalently, the same structure can be thought of a simplicial complex with some missing faces. These missing faces are a result of edges of $T$ which, were their lengths reduced to zero, would create new vertices whose stabilisers were not in $\mathcal{G}$, and hence the resulting tree could not lie in $\mathcal{O}$.)

Let $\Delta$ be an open simplex in $\mathcal{O}_1$. We write $\overline{\Delta}$ to denote the closure of $\Delta$ in $\mathcal{O}_1$. Note that this is not, in general, a closed simplex.

\begin{defn}(Adapted from \cite[p.19, Def. 5.1]{francavigliamartino2018b})
	
	Let $T,S\in\mathcal{O}_1$. A \emph{simplicial path} between $T$ and $S$ is given by:
	
	\begin{itemize}
		\item[(i)] A finite sequence of points $T=T_0,T_1,\ldots,T_k = S\in\mathcal{O}_1$ such that $\forall i = 1\ldots k$ there is a simplex $\Delta_i$ such that $X_{i-1}$ and $X_i$ both lie in $\overline{\Delta_i}$.
		
		\item[(ii)] Euclidean segments $\overline{X_{i-i}X_i}\subseteq \overline{\Delta_i}$. (Here Euclidean segment refers to the coordinates $(x_1,\ldots,x_n)$ on $\overline{\Delta_i}$)
	\end{itemize}
\end{defn}

\begin{defn}
	We say that a set $\chi\subseteq\mathcal{O}_1$ is \emph{connected by simplicial paths} if for any $x,y\in \chi$, there is a simplicial path between $x$ and $y$ which is contained entirely in $\chi$.
\end{defn}


\begin{lem}\label{lem:disconnected_centres}
	A simplicial path in $\mathcal{O}_1$ which only passes through the centres of simplices is a single point.
\end{lem}
\begin{proof}
	Any such simplicial path must begin at the centre of an open simplex in $\mathcal{O}_1$. Observe that any nontrivial Euclidean segment which begins at the centre of a simplex must pass through a point which does not lie at the centre of a simplex. Thus the entire simplicial path is trivial.
\end{proof}

We can now state the main theorem of this paper.

\MSonepoint

\begin{proof}
	By Theorem \ref{thm:irreduciblerealisesdisplacement}, $\text{Min}(\a)$, and hence $\text{Min}_1(\a)$, is non-empty.
	
	Let $T\in \text{Min}_1(\a)$. Then by Theorem \ref{thm:isom_top_rep_exists} there exists an equivariant isometry $f:T\to \a T$. By Theorem \ref{thm:permuting_edge_orbits}, $f$ cyclically permutes the edges of $T$, which means that all the edges in $T$ must have the same length, and hence $T$ must lie at the centre of an open simplex in $\mathcal{O}_1$. Thus  $\text{Min}_1(\a)$ is a subset of the set of simplex centres.
	
	It is shown in \cite[p.19, Cor 5.4]{francavigliamartino2018b} that $\text{Min}_1(\a)$ is connected by simplicial paths. However, by Lemma \ref{lem:disconnected_centres}, a simplicial path in $\mathcal{O}_1$ which only passes through the centres of simplices is a single point. It follows that $\text{Min}_1(\a)$ is a single point.
\end{proof}

\begin{cor}
	Let $\a\in\text{Out}(G,\mathcal{G})$ be irreducible and displacement 1. Then $\text{Min}(\a) = \text{Fix}(\a)$ is a single line.
\end{cor}
\begin{proof}
	Follows directly from Theorem \ref{thm:min_set_single_point} and the fact that\\ $\text{Min}(\a) = \{\mu T\in\mathcal{O} \mid T\in \text{Min}_1(\a), \mu > 0 \}$
\end{proof}

There exists a space similar to $\mathcal{O}_1$ called the projectivized space $\mathcal{PO}$, where instead of taking the covolume one subspace of $\mathcal{O}$, one takes a quotient space of $\mathcal{O}$ by identifying all $G$-trees in the sets $\{(T,\lambda d_T,\cdot)\mid \lambda\in \R\}$ for each $T\in\mathcal{O}$.

When choosing a $G$-tree to represent a point in $\mathcal{PO}$, one usually takes the unique covolume one $G$-tree. In this way, we can construct a natural bijection between $\mathcal{O}_1$ and $\mathcal{PO}$.

The displacement $\lambda_\a$ of an automorphism is invariant under rescaling of the metrics $d_T$, hence we can define the minimally displaced set $\text{Min}_\mathcal{P}(\a)$ in $\mathcal{PO}$. The set $\text{Min}_1(\a)$ is a set of representatives for $\text{Min}_\mathcal{P}(\a)$, thus Theorem \ref{thm:min_set_single_point} shows that $\text{Min}_\mathcal{P}(\a)$ is also a single point.


\begin{thebibliography}{10}
	
	\bibitem{bestvinafeighnhandel}
	Mladen Bestvina, Mark Feighn, and Michael Handel.
	\newblock The {T}its alternative for {${\rm Out}(F_n)$}. {I}. {D}ynamics of
	exponentially-growing automorphisms.
	\newblock {\em Ann. of Math. (2)}, 151(2):517--623, 2000.
	
	\bibitem{bestvinahandel}
	Mladen Bestvina and Michael Handel.
	\newblock Train tracks and automorphisms of free groups.
	\newblock {\em The Annals of Mathematics}, 135(1):1--51, jan 1992.
	
	\bibitem{cohen}
	Daniel~E. Cohen.
	\newblock {\em Combinatorial group theory: a topological approach}, volume~14
	of {\em London Mathematical Society Student Texts}.
	\newblock Cambridge University Press, Cambridge, 1989.
	
	\bibitem{cullermorgan}
	Marc Culler and John~W. Morgan.
	\newblock Group actions on {${\bf R}$}-trees.
	\newblock {\em Proceedings of the London Mathematical Society},
	s3-55(3):571--604, nov 1987.
	
	\bibitem{cullervogtmann}
	Marc Culler and Karen Vogtmann.
	\newblock Moduli of graphs and automorphisms of free groups.
	\newblock {\em Inventiones Mathematicae}, 84(1):91--119, feb 1986.
	
	\bibitem{dicksventura}
	Warren Dicks and Enric Ventura.
	\newblock Irreducible automorphisms of growth rate one.
	\newblock {\em Journal of Pure and Applied Algebra}, 88(1-3):51--62, 1993.
	
	\bibitem{forester}
	Max Forester.
	\newblock Deformation and rigidity of simplicial group actions on trees.
	\newblock {\em Geometry and Topology}, 6:219--267, 2002.
	
	\bibitem{francavigliamartino2015}
	Stefano Francaviglia and Armando Martino.
	\newblock Stretching factors, metrics and train tracks for free products.
	\newblock {\em Illinois Journal of Mathematics}, 59(4):859--899, 2015.
	
	\bibitem{francavigliamartino2018a}
	Stefano Francaviglia and Armando Martino.
	\newblock Displacements of automorphisms of free groups {I}: {D}isplacement
	functions, minpoints and train tracks.
	\newblock {\em Transactions of the American Mathematical Society},
	374(5):3215--3264, 2021.
	
	\bibitem{francavigliamartino2018b}
	Stefano Francaviglia and Armando Martino.
	\newblock Displacements of automorphisms of free groups {II}: {C}onnectivity of
	level sets and decision problems.
	\newblock {\em Transactions of the American Mathematical Society},
	375(4):2511--2551, 2022.
	
	\bibitem{francavigliamartino2021}
	Stefano Francaviglia, Armando Martino, and Dionysios Syrigos.
	\newblock The minimally displaced set of an irreducible automorphism of {$F_N$}
	is co-compact.
	\newblock {\em Archiv der Mathematik}, 116(4):369--383, 2021.
	
\end{thebibliography}

\end{document}